\documentclass{amsart}
\usepackage{amsmath}
\usepackage{url}

\usepackage{amssymb,latexsym}
\usepackage[dvips]{graphicx}

\theoremstyle{plain}

\newtheorem{theorem}{Theorem} [section]
\newtheorem{lemma}{Lemma} [section]
\newtheorem{proposition}{Proposition} [section]
\newtheorem{corollary}{Corollary} [section]
\newtheorem{definition}{Definition} [section]

\newtheorem{remark}{Remark}[section]

\begin{document}

\date{} 

\title[Domain deformations and eigenvalues of the Laplacian]
{Domain deformations and eigenvalues of the Dirichlet Laplacian in a Riemannian manifold}
\author{Ahmad El Soufi and Sa\"{\i}d Ilias}

\address{Universit\'e de Tours, Laboratoire de Math\'ematiques
et Physique Th\'eorique, UMR-CNRS 6083, Parc de Grandmont, 37200
Tours, France} 
\email{elsoufi@univ-tours.fr, ilias@univ-tours.fr}

\begin{abstract}
For any bounded regular domain $\Omega$ of a real analytic Riemannian manifold
$M$, we denote by $\lambda_{k}(\Omega)$ the $k$-th eigenvalue of the Dirichlet Laplacian
of $\Omega$. In this paper, we consider $\lambda_k$ and as a
functional upon the set of domains of fixed volume in $M$. We
introduce and investigate a natural notion of critical domain for
this functional. In particular, we obtain necessary and
sufficient conditions for a domain to be critical, locally
minimizing or locally maximizing for $\lambda_k$. These
results rely on Hadamard type variational formulae that we
establish in this general setting.

As an application, we obtain a characterization of critical
domains of the trace of the heat kernel under Dirichlet 
boundary conditions.

\end{abstract}

\maketitle

\section {Introduction}

Isoperimetric eigenvalue problems constitute one of the main
topics in spectral geometry and shape optimization. Given a
Riemannian manifold $M$, a natural integer $k$ and a positive
constant $V$, the problem is to optimize the $k$-th eigenvalue of
the Dirichlet Laplacian, considered as a functional
upon the set of all bounded domains of volume $V$ of $M$.

The first result in this subject is the famous Faber-Krahn Theorem
\cite{F, Kr}, originally conjectured by Rayleigh, stating that
Euclidean balls minimize the first eigenvalue of the Dirichlet
Laplacian among all domains of given volume. Extensions of
this classical result to higher order eigenvalues,
combinations of eigenvalues as well as domains of other Riemannian
manifolds or subjected to other types of constraints, have been
obtained during the last decades and a very rich literature is
devoted to this subject (see for instance \cite{A, AB, B, BBGa,
BM, C, E, EI1, H, N, OPS, R1, Sh, Sz, W} and the references therein).

A fundamental tool in the proof of many results concerning the
first Dirichlet eigenvalue is the following variation formula,
known as Hadamard's formula \cite{Ha, GaSc, Sc1, Sc2}:
$${\frac{d}{d\varepsilon}} \lambda_1(\Omega_\varepsilon)\big|_{\varepsilon=0}=-
\int_{\partial\Omega_0} v \left({\frac{\partial \phi}{ \partial
\nu}}\right)^2 d\sigma,$$ where $\lambda_1(\Omega_\varepsilon)$
stands for the first Dirichlet eigenvalue of the domain
$\Omega_\varepsilon$, $\frac{\partial \phi}{\partial \nu}$ denotes
the normal derivative of the first normalized eigenfunction $\phi$
of the Dirichlet Laplacian on $\Omega_0$ and $v$ is the normal
displacement of the boundary induced by the deformation. This
formula shows that a necessary and sufficient condition for a
domain $\Omega \subset {\mathbb R}^n$ to be critical for the
Dirichlet first eigenvalue functional under fixed volume
variations, is that its first Dirichlet eigenfunctions are
solutions of the following overdetermined problem:

$$\left\{
\begin{array}{l}
\Delta {\phi} = \lambda_1(\Omega) {\phi}\ \hbox{in}\ \Omega \\
\\
{\phi}=0 \ \ \hbox{on} \ \partial\Omega\\
\\
|{\partial \phi\over \partial\nu}| =c \ \ \hbox{on}\
\partial\Omega,
\end{array}
\right.$$ for some constant $c$. Since a first Dirichlet
eigenfunction does not change sign in $\Omega$, it follows from
the well known symmetry result of Serrin \cite{S} that $\phi$ is
radial and $\Omega$ is a round ball. Therefore, Euclidean balls
are the only critical domains of the Dirichlet first eigenvalue
functional under fixed volume deformations.

Notice that Hadamard's formula remains valid for any higher order
eigenvalue $\lambda_k$ as far as $\lambda_k(\Omega)$ is simple.
Nevertheless, when $\lambda_k(\Omega)$ is degenerate, a
differentiability problem arises and our first aim in this paper
(see Section 3) is to introduce, in spite of this
non-differentiability problem, a natural and simple notion of
critical domain.

Indeed, using perturbation theory of unbounded self-adjoint
operators in Hilbert spaces, we will see that, for any deformation
$\Omega_\varepsilon$, analytic in $\varepsilon$, of a domain
$\Omega$ of a real analytic Riemannian manifold $M$, and any natural integer
$k$, the function $\varepsilon\mapsto \lambda_k(\Omega_\varepsilon
)$ admits a left sided and a right sided derivatives at
$\varepsilon =0$. Of course, when $\Omega$ is a local extremum of $\lambda_k$, these derivatives have opposite signs. This suggests us
to define critical domains of $\lambda_k$ to be
the domains $\Omega$ such that, for any analytic volume-preserving
deformation $\Omega_\varepsilon$ of $\Omega$, the right sided and
the left sided derivatives of $\lambda_{k}(\Omega_\varepsilon)$
at $\varepsilon = 0$ have opposite signs. That is,

$${d \over d\varepsilon}\lambda_{k}(\Omega_\varepsilon){\big|_{\varepsilon=0^+}}
 \times\ {d \over d\varepsilon}\lambda_{k}(\Omega_\varepsilon){\big|_{\varepsilon=0^-}}
\leq 0.$$
which means that $\lambda_{k} (\Omega_\varepsilon)\leq
\lambda_{k} (\Omega)+ o({\varepsilon})$ or
$\lambda_{k}(\Omega_\varepsilon) \geq
\lambda_{k}(\Omega_\varepsilon) + o({\varepsilon}) $ as
$\varepsilon\to 0$.

\medskip

After giving, in Section 2, a general Hadamard type variation
formula, we derive, in Section 3, necessary and sufficient
conditions for a domain $\Omega$ of the Riemannian manifold $M$ to
be critical for the $k$-th Dirichlet eigenvalue functional under volume-preserving domain deformations.

For instance, we show that (Theorem \ref{lambda_k}) if $\Omega$ is
a critical domain of the $k$-th Dirichlet eigenvalue under volume-preserving domain deformations, then there exists a family of
eigenfunctions $\phi_1,\ldots, \phi_m$ satisfying the following
system:

\begin{equation}\label{1}
\left\{
\begin{array}{l}
\Delta {\phi_i} = \lambda_k (\Omega)\, {\phi_i} \ \hbox {in} \ \Omega, \, \ \forall i \leq m, \\
\\
{\phi_i}=0 \ \hbox {on} \ \partial \Omega,\ \forall i \leq m,\\
\\
\sum_{i=1}^{m} \left({\partial \phi_i\over \partial \nu}\right)^2 =1 \ \hbox {on} \ \partial \Omega.
\end{array}
\right.
\end{equation}
Moreover, this necessary condition is also sufficient when either
$\lambda_k (\Omega)> \lambda_{k-1}(\Omega)$ or $\lambda_k (\Omega)< \lambda_{k+1}(\Omega)$,
which means that $\lambda_k(\Omega)$ corresponds to the first one or the last one
in a cluster of equal eigenvalues.
On the other hand, we prove that if $\lambda_k (\Omega)>
\lambda_{k-1}(\Omega)$ (resp. $\lambda_k (\Omega)<
\lambda_{k+1}(\Omega)$) and if $\Omega\subset M$ is a local
minimizer (resp. maximizer) of the $k$-th Dirichlet eigenvalue
functional under volume-preserving domain deformations, then
$\lambda_k (\Omega)$ is simple and the absolute value of the
normal derivative of its corresponding eigenfunction is constant
along the boundary $\partial \Omega$ (see Theorem \ref{lambda_k
min}).

The last section deals with the trace of the heat kernel under
Dirichlet boundary conditions defined for a domain $\Omega \subset M$ by
$$Y_{\Omega}(t) = \int_\Omega H(t,x,x) v_g=\sum_{k\ge 1} e^{- \lambda_k(\Omega) t}, $$
where $H$ is the fundamental
solution of the heat equation in $\Omega$ under Dirichlet boundary conditions. Indeed, Luttinger \cite{L} proved an isoperimetric Faber-Krahn
like result for $Y(t)$ considered as a
functional upon the set of bounded Euclidean domains, that is, for any bounded domain $\Omega\subset
\mathbb{R}^n$ and any $t>0$, one has $Y_{\Omega}(t)\le
Y_{\Omega^*}(t)$, where $\Omega^*$ is an Euclidean ball whose
volume is equal to that of $\Omega$.

For any smooth deformation $\Omega_\varepsilon$ of $\Omega$, the corresponding
heat trace function $Y_\varepsilon (t) $ is always
differentiable w.r.t. $\varepsilon$ and the domain $\Omega$ will be said critical
for the trace of the heat kernel under the Dirichlet boundary
condition at time $t$ if, for any volume-preserving deformation
$\Omega_\varepsilon$ of $\Omega$, we have
$${d\over {d \varepsilon }} Y_\varepsilon (t)\big| _{\varepsilon=0} =0$$
After giving the first variation formula for this functional
(Theorem \ref{varnoy}), we show that a necessary and sufficient
condition for a domain $\Omega$ to be critical for the trace of
the heat kernel under Dirichlet boundary
condition at time $t$ is that the Laplacian of the function
$x\mapsto H(t,x,x) $ must be constant
along the boundary $\partial\Omega$ (Corollary \ref{criticH}).

Using Minakshisundaram-Pleijel asymptotic expansion of $Y(t)$, one can derive necessary conditions for a domain to be
critical for the trace of the heat kernel under Dirichlet boundary condition at any time $t>0$. For instance, we
show that the boundary of such a domain necessarily has constant
mean curvature (Theorem \ref{hconst}).

 Thanks to Alexandrov type results (see \cite{Al, MR}), one deduces that
 when the ambient space $M$ is Euclidean,
 Hyperbolic or a standard hemisphere, then
 geodesic balls are the only critical domains of the trace of the heat kernel
 under Dirichlet boundary condition at any time $t>0$ (Corollary \ref{balls}).


\section {Hadamard type variation formulae}

Let $\Omega$ be a regular bounded domain of a Riemannian oriented manifold
$(M,g)$. We will denote by $\bar g$ the metric induced by $g$ on the boundary
$\partial \Omega$ of $\Omega$.
Let us start with the following general formula.

\begin{proposition}\label{varmetric} Let $(g_\varepsilon)$ be a differentiable variation
of the metric $g$. Let $\phi_\varepsilon\subset {\mathcal
C}^\infty (\Omega)$ be a differentiable family of functions and
$\Lambda_\varepsilon$ a differentiable family of real numbers such
that, $\forall \varepsilon$,
$||\phi_\varepsilon||_{L^2{(\Omega,g_\varepsilon)}}=1$ and
$$\left\{
\begin{array}{l}
\Delta_{g_\varepsilon} \phi_\varepsilon = {\Lambda}_\varepsilon \phi_\varepsilon \;\; \hbox{in} \; \Omega\\
\\
 \phi_\varepsilon =0 \;\; \hbox{on} \; \partial \Omega.
\end{array}
\right.$$
Then,
\begin{eqnarray}
\nonumber{} {d\over{d\varepsilon}}{\Lambda}_\varepsilon\big|_{\varepsilon=0}
&=& \int_\Omega \phi_0\Delta'\phi_0 v_g \\
\nonumber {} &=& -\int_\Omega \langle d\phi_0 \otimes d\phi_0 +{1\over 4}\Delta \phi_0^2 \ g , h\rangle v_g,
\end{eqnarray}
where
$h:={d\over{d\varepsilon}}g_\varepsilon\big|_{\varepsilon=0}$,
$\Delta':={d\over{d\varepsilon}}\Delta_{g_\varepsilon}\big|_{\varepsilon=0}$
and $\langle,\rangle$ is the inner product induced by $g$ on the
space of covariant tensors.
\end{proposition}

\begin{proof}
For simplicity, let us introduce the following notations:
$\lambda:=\Lambda_0$, $\phi:=\phi_0$, $\phi':=
{d\over{d\varepsilon}}\phi_\varepsilon\big|_{\varepsilon=0}$ and
$\Lambda':={d\over{d\varepsilon}}\Lambda_\varepsilon\big|_{\varepsilon=0}$.

Differentiating the two sides of the equality
$\Delta_{g_\varepsilon} \phi_\varepsilon = {\Lambda}_\varepsilon
\phi_\varepsilon$ we obtain
$$\Delta'\phi + \Delta \phi'=\Lambda'\phi+\Lambda\phi'.$$
After multiplication by $\phi$ and integration we get
$$\int_\Omega \phi \Delta'\phi v_g+\int_\Omega \phi \Delta\phi' v_g
= \Lambda'+ \lambda \int_\Omega \phi\phi' v_g.$$
Integration by parts gives
$$\int_\Omega \phi \Delta\phi' v_g= \lambda \int_\Omega \phi\phi' v_g
+ \int_{\partial \Omega} ({\partial \phi\over \partial\nu} \phi'
-{\partial \phi'\over \partial\nu} \phi) v_{\bar g}.$$
Thus,
$$
\Lambda'=\int_\Omega \phi \Delta'\phi v_g
+ \int_{\partial \Omega} (\phi'{\partial \phi\over \partial\nu}
-\phi{\partial \phi'\over \partial\nu} ) v_{\bar g}.
$$

It is clear that the boundary integral in this last equation
vanishes (since $\phi_\varepsilon=0$ on $\partial\Omega$).\\
In conclusion, we have

\begin{equation}\label{3}
 \Lambda'= \int_\Omega \phi \Delta'\phi v_g.
\end{equation}

Now, the expression of $\Delta'$ is given by (see \cite{B})
\begin{equation}\label{4}
\Delta'\phi=\langle D^2\phi,h\rangle - \langle d\phi,\delta
h+{1\over2}d\tilde{h}\rangle,
\end{equation}
where $\tilde{h}$ is the trace of $h$ w.r.t. $g$ (that is $\tilde{h}=\langle g,h\rangle$).
Integration by parts yields

\begin{eqnarray}\label{5}
 \;\;\; \int_\Omega \phi \langle d\phi,\delta h\rangle v_g &=& {1\over 2}\int_\Omega \langle D^2\phi^2,h\rangle v_g \\
 \nonumber {} &=& \int_\Omega \langle d\phi\otimes d\phi + \phi D^2\phi,h\rangle v_g
\end{eqnarray}

and

\begin{eqnarray}\label{6}
\int_\Omega \phi \langle d\phi,d\tilde h\rangle v_g
&=&{1\over 2} \int_\Omega \tilde h \Delta \phi^2 v_g 
\end{eqnarray}
Combining (\ref{3}), (\ref{4}), (\ref{5}) and (\ref{6}) we obtain
$$\Lambda' = -\int_\Omega \langle d\phi\otimes d\phi + {1\over4}\Delta\phi^2 g ,h\rangle v_g$$
which completes the proof of the proposition.
\end{proof}

In the particular case of domain deformations, Proposition
\ref{varmetric} gives rise to the following variation
formulae.

\begin{corollary}\label{vardir} Let $\Omega_\varepsilon = f_\varepsilon
(\Omega)$ be a deformation of $\Omega$. Let $\phi_\varepsilon \in
{\mathcal C}^\infty (\Omega_\varepsilon)$ and
${\Lambda}_\varepsilon \in {\bf R}$ be two differentiable curves
such that, $\forall \varepsilon$, $||\phi_\varepsilon||_{L^2(\Omega_{\varepsilon},g)}=1$ 
and
$$\left\{
\begin{array}{l}
\Delta \phi_\varepsilon = {\Lambda}_\varepsilon \phi_\varepsilon \;\; \hbox{in\;} \Omega_\varepsilon\\
\\
 \phi_\varepsilon =0 \;\; \hbox{on} \; \partial \Omega_\varepsilon.
\end{array}
\right.$$
Then,
$${d\over d\varepsilon}\Lambda{_\varepsilon}\big|_{\varepsilon=0}=-
\int_{\partial\Omega} v \left({\partial \phi \over \partial
\nu}\right)^2 v_{\bar g},$$ where $\phi=\phi_0$ and
${v=g\left({d\over
d\varepsilon}f_\varepsilon\big|_{\varepsilon=0},\nu\right)}$ is
the normal component of the variation vector field of the
deformation $\Omega_\varepsilon$.
\end{corollary}
\begin{proof}
Let us apply Proposition \ref{varmetric} with
$g_\varepsilon=f_\varepsilon^*g$ and $\bar
\phi_\varepsilon=\phi_\varepsilon\circ f_\varepsilon$. Indeed, one
can easily check that $\|\bar
\phi_\varepsilon\|_{L^2(\Omega,g_\varepsilon)}=1$,
$\Delta_{g_\varepsilon}\bar
\phi_\varepsilon=\Lambda_\varepsilon\bar \phi_\varepsilon$ in
$\Omega$ and $\bar \phi_\varepsilon=0$ on $\partial\Omega$. Hence,
\begin{equation}\label{7}
{d\over d\varepsilon} \Lambda{_\varepsilon}\big|_{\varepsilon=0}=-
\int_{\Omega} \langle d\phi \otimes d\phi +{1\over 4}\Delta \phi^2
\ g , h\rangle v_g
\end{equation}
with $\phi:=\phi_0=\bar\phi_0$ and $h={d\over
d\varepsilon}f_\varepsilon^*g\big|_{\varepsilon=0}={\mathcal L}_V
g$, where ${\mathcal L}_V g$ is the Lie derivative of $g$ w.r.t.
the vector field $V={d\over
d\varepsilon}f_\varepsilon\big|_{\varepsilon=0}$.

Using the expression of ${\mathcal L}_V g$ in terms of the
covariant derivative $\nabla V$ of $V$ and integrating by parts,
we obtain
\begin{eqnarray}
 \nonumber {} &\displaystyle{}\int_\Omega &\langle d\phi \otimes d\phi , {\mathcal L}_V g \rangle v_g =\int_\Omega {\mathcal L}_V g ( \nabla\phi,\nabla\phi) v_g =2\int_\Omega \langle \nabla_{\nabla\phi}V,\nabla\phi\rangle v_g \\
 \nonumber {} &=& \int_\Omega \hbox{div} (\langle V,\nabla\phi \rangle \nabla\phi) v_g + 2\int_\Omega \langle V,\nabla\phi \rangle \Delta\phi v_g -2\int_\Omega D^2\phi (V,\nabla\phi) v_g\\
\nonumber {} &=&2\int_{\partial\Omega} \langle V,\nabla\phi \rangle{\partial \phi \over \partial \nu}\ v_{\bar g}+\lambda \int_\Omega \langle V,\nabla\phi^2 \rangle v_g -2\int_\Omega D^2\phi (V,\nabla\phi) v_g,
\end{eqnarray}
with $\lambda:= \Lambda_0$, and
\begin{eqnarray}
 \nonumber {}{1\over4} \int_\Omega \Delta\phi^2 \langle g,{\mathcal L}_V g \rangle v_g
 & =& {1\over2}\int_\Omega \Delta\phi^2 \hbox{div} V v_g \\
 \nonumber {} &=&\lambda \int_\Omega \phi^2 \hbox{div} V v_g - \int_\Omega |\nabla \phi|^2\hbox{div} V v_g \\
\nonumber {} &=& \int_\Omega (- \lambda \langle V,\nabla\phi^2 \rangle +2 D^2\phi (V,\nabla\phi)) v_g\\
\nonumber {} &+&\int_{\partial\Omega} (\lambda\phi^2-|\nabla \phi|^2) \langle V,\nu \rangle v_{\bar g}.
\end{eqnarray}
Replacing in (\ref{7}), we get
$$ {d\over d\varepsilon} \Lambda{_\varepsilon}{\big|_{\varepsilon=0}}=
\int_{\partial\Omega}\{- 2 \langle V,\nabla \phi \rangle {\partial \phi \over \partial \nu}
+ \langle V,\nu \rangle |\nabla\phi|^2
 - \lambda \langle V,\nu \rangle \phi^2\} v_{\bar g}.$$
Since $\phi$ is identically zero on the boundary, we have at any
point of $\partial \Omega$, $\nabla\phi={\partial \phi \over
\partial \nu} \nu$. In particular, $
|\nabla\phi|^2=\left({\partial \phi \over \partial \nu}\right)^2$
and $ \langle V,\nabla\phi \rangle= \langle V,\nu \rangle
{\partial \phi \over \partial \nu}=v{\partial \phi \over \partial
\nu}$. Thus,
$${d\over d\varepsilon}\ \Lambda{_\varepsilon}{\big|_{\varepsilon=0}}=-
\int_{\partial\Omega} v \left({\partial \phi \over \partial \nu}\right)^2 v_{\bar g}.$$

\end{proof}

\bigskip

\section{Critical domains}

Throughout this section, the ambient Riemannian manifold $(M,g)$ is assumed to be real analytic.
\subsection{Preliminary results and definitions}
Let $\Omega$ be a regular bounded domain of a Riemannian manifold
$(M,g)$. An analytic deformation $(\Omega_\varepsilon)$ of
$\Omega$ is given by an analytic 1-parameter family of
diffeomorphisms $f_\varepsilon: \Omega \rightarrow
\Omega_\varepsilon$ such that $f_\varepsilon (\partial \Omega)=
\partial \Omega_\varepsilon$ and $f_0 = Id$. Such a deformation is
called volume-preserving if the Riemannian volume of
$\Omega_\varepsilon$ w.r.t the metric $g$ does not depend on
$\varepsilon$.

The spectrum of the Laplace operator $\Delta_g$ under the
Dirichlet boundary condition will be denoted
$$Sp_{_D}\ (\Delta_g ,\Omega_\varepsilon)= \{\ \lambda_{1,\varepsilon}\
< \lambda_{2,\varepsilon} \ \leq \cdots \leq\ \lambda_{k,\varepsilon}\
\ \uparrow +\ \infty\ \}$$
The functions $\varepsilon \mapsto \lambda_{k,\varepsilon}$ is continuous but not
differentiable in general, excepting $\lambda_{1,\varepsilon}$
which is always differentiable since it is simple. However, as we
will see hereafter, the general perturbation theory of unbounded
self-adjoint operators enables us to show that the function
$\lambda_{k,\varepsilon}$ admits a
right sided and a left sided derivatives at $\varepsilon = 0$. In
all the sequel, a family of functions $\phi_\varepsilon \in
{\mathcal C}^\infty (\Omega_\varepsilon)$ will be said
differentiable (resp. analytic) w.r.t $\varepsilon$, if that is
the case for $\phi_\varepsilon\circ f_\varepsilon \in {\mathcal
C}^\infty (\Omega)$.

\begin{lemma}\label{perturbation} Let $\lambda \in Sp_{_D}(\Delta_g , \Omega)$
be an eigenvalue of multiplicity $p$ of the Dirichlet Laplacian in
$\Omega$. For any analytic deformation $\Omega_\varepsilon$ of
$\Omega$, there exist $p$ families $(\Lambda_{i, \varepsilon})_{i
\leq p}$ of
real numbers and $p$ families $(\phi_{i, \varepsilon})_{i \leq p}
\subset {\mathcal C}^\infty (\Omega_\varepsilon)$ of functions, depending analytically on
$\varepsilon$ and satisfying, $\forall\varepsilon \in
(-\varepsilon_0 , \varepsilon_0)$ and $\forall i \in
\{1,\ldots,p\}$,

\begin{enumerate}

\item[(a)] $\Lambda_{i,0}= \lambda$
\item[(b)] the family $\{\phi_{1,\varepsilon},
\cdots, \phi_{p,\varepsilon}\}$ is orthonormal in
$L^2(\Omega_\varepsilon,g)$ \item[(c)] $\left\{ \begin{array}{l}
\Delta\ \phi_{i,\varepsilon}\ =\Lambda_{i,\varepsilon} \phi_{i,\varepsilon} \ {\mbox in} \ \Omega_\varepsilon \\
\\
\phi_{i,\varepsilon}=0 \ {\mbox on} \ {\partial\Omega_\varepsilon}
\end{array}\right.$
\end{enumerate}
\end{lemma}

 The proof is based on perturbation theory of
unbounded self-adjoint operators in Hilbert spaces. Results
concerning the differentiability of eigenvalues and eigenvectors have been first
obtained by Rellich \cite{Re} and then by Kato \cite{K} in the
analytic case. Many results were also obtained under weaker
differentiability conditions (see for instance  \cite{KM0, KM} for recent contributions in this subject). Nevertheless, even a smooth curve $\varepsilon\mapsto P_\varepsilon$ of self-adjoint operators may lead to noncontinuous eigenvectors w.r.t. $\varepsilon$ (see Rellich's example \cite[chap II, Example 5.3]{K}). Since we need to differentiate eigenvectors w.r.t. $\varepsilon$, we imposed analyticity assumptions in order to get analytic curves of operators.   

\begin{proof}[Proof of Lemma \ref{perturbation}] In order to get into the framework
of perturbation theory, we first need to modify our operators so
that they all have the same domain. Indeed, for any $\varepsilon$
we set $g_\varepsilon = f_\varepsilon^* g$ and denote $\Delta
_\varepsilon$ the Laplace operator of $(\Omega,g_\varepsilon)$.
Clearly, we have
$$Sp_{_D} (\Delta_g, \Omega_\varepsilon) = Sp_{_D} (\Delta_\varepsilon,\Omega).$$
Notice that since $f_\varepsilon$ depends analytically on $\varepsilon$ and that $g$ is real analytic, the curves $\varepsilon\mapsto g_\varepsilon$ and, hence, $\varepsilon\mapsto \Delta_\varepsilon$ are analytic w.r.t. $\varepsilon$. 

The operator $\Delta_\varepsilon$ is symmetric w.r.t. the inner product in
$L^2 (\Omega, g_\varepsilon)$, but not necessarily w.r.t. the
inner product in $L^2 (\Omega,g)$. Therefore, we need to introduce a
conjugation as follows. Let $U_\varepsilon: L^2(\Omega,g)
\rightarrow L^2(\Omega,g_\varepsilon)$ be the unitary isomorphism
given by
$$U_\varepsilon: v \mapsto \left({|g|\over |g_\varepsilon|}\right)^{1\over 4} v,$$
where $|g| = \det (g_{ij})$ is the determinant of the matrix $(g_{ij})$
of the components of $g$ in a local coordinate system.
We define the operator $P_\varepsilon$ to be
$$P_\varepsilon = U_\varepsilon^{-1} \circ \Delta_\varepsilon\circ U_\varepsilon.$$
Therefore, we have $Sp_{_D} (P_\varepsilon,\Omega)= Sp_{_D} (\Delta_\varepsilon, \Omega)$
and, if $v_\varepsilon \in {\mathcal C}^\infty (\Omega)$ is an eigenfunction of
$P_\varepsilon$, then $\phi_\varepsilon = U_\varepsilon (v_\varepsilon) \circ
f_\varepsilon^{-1}\ \in \ {\mathcal C}^\infty (\Omega_\varepsilon)$
is an eigenfunction of $\Delta_g$ with the same eigenvalue. Again, since $\forall \varepsilon$, $(M,g_\varepsilon)$ is real analytic, the curves $\varepsilon\mapsto U_\varepsilon$ and $\varepsilon\mapsto P_\varepsilon$ are analytic.   The
result of the lemma then follows
from the Rellich-Kato theory applied to $\varepsilon
\mapsto P_\varepsilon$.

\end{proof}

\bigskip

Now, let us fix a positive integer $k$ and let
$\Lambda_{1,\varepsilon}, \ldots, \Lambda_{p,\varepsilon}$ be the
family of eigenvalues associated with $\lambda_k$ by Lemma
\ref{perturbation}. Using the continuity of
$\lambda_{k,\varepsilon}$ and the analyticity of
$\Lambda_{i,\varepsilon}$ w.r.t. $\varepsilon$, we can easily see
that there exist two integers $i \leq p$ and $j \leq p$ such that
$$ \lambda_{k, \varepsilon} =\left\{ \begin{array}{l}
\Lambda_{i,\varepsilon}\ \hbox{if}\ \varepsilon \leq 0\\
\\
\Lambda_{j,\varepsilon}\ \hbox{if}\ \ \varepsilon \geq 0.
\end{array}\right.$$
Hence, $\lambda_{k,\varepsilon}$ admits a left sided and a right
sided derivatives with
$${d\over d\varepsilon} \lambda_{k,\varepsilon} {\big|_{\varepsilon=0^+}}\ \
= {d\over d\varepsilon} \Lambda_{j,\varepsilon} {\big|_{\varepsilon=0}}$$
and
$${d\over d\varepsilon} \lambda_{k,\varepsilon}{\big|_{\varepsilon=0^-}}\ \
= {d\over d\varepsilon} \Lambda_{i,\varepsilon}{\big|_{\varepsilon=0}}.$$

\begin{definition}\label{def1} The domain $\Omega$ is said to be "critical" for the
$k$-$th$ eigenvalue of Dirichlet problem 
if, for any analytic volume-preserving deformation
$\Omega_\varepsilon$ of $\Omega$, the right sided and the left
sided derivatives of $\lambda_{k,\varepsilon}$ at $\varepsilon = 0$ have opposite signs.
That is,
$${d \over d\varepsilon}\lambda_{k,\varepsilon}{\big|_{\varepsilon=0^+}}
 \times\ {d \over d\varepsilon}\lambda_{k,\varepsilon}{\big|_{\varepsilon=0^-}}
\leq 0.$$ 
\end{definition}

It is easy to see that
$${d\over d\varepsilon}\lambda_{k,\varepsilon}{\big|_{\varepsilon =0^+}} \leq 0 \leq
{d \over d\varepsilon}\lambda_{k,\varepsilon}{\big|_{\varepsilon=0^-}}\Longleftrightarrow
\lambda_{k,\varepsilon} \leq \lambda_{k,0} + o (\varepsilon)$$
and
$${d\over d\varepsilon}\lambda_{k,\varepsilon}{\big|_{\varepsilon=0^-}} \leq 0 \leq
{d \over d\varepsilon}\lambda_{k,\varepsilon}{\big|_{\varepsilon=0^+}}\Longleftrightarrow
\lambda_{k,\varepsilon} \geq \lambda_{k,0} + o (\varepsilon).$$\\
Therefore, the domain $\Omega$ is critical for the $k$-th
eigenvalue of Dirichlet problem if and only if one of the
following inequalities holds:
$$\lambda_{k,\varepsilon} \leq \lambda_{k,0} + o ({\varepsilon})$$
$$\lambda_{k,\varepsilon} \geq \lambda_{k,0} + o ({\varepsilon}) .$$\\

\begin{remark}\label{rem1} Suppose that for an integer $k$ we
have $\lambda_k < \lambda_{k+1}$, then, for sufficiently small
$\varepsilon$, we will have $\lambda_{k,\varepsilon}\displaystyle
=\max_{i \leq p} \Lambda_{i,\varepsilon}$, where
$\Lambda_{1,\varepsilon},\ldots,\Lambda_{p,\varepsilon}$ are the
eigenvalues associated to $\lambda_k$ by Lemma \ref{perturbation}
(indeed, $\Lambda_{i,0}=\lambda_k < \lambda_{k+1}$ for any $1\le
i\le p$). Hence, ${d\over d\varepsilon}\
\lambda_{k,\varepsilon}{|_{\varepsilon = 0^-}} \leq {d\over
d\varepsilon}\lambda_{k,\varepsilon}{|_{\varepsilon = 0^+}}$. In
particular, $\Omega$ is critical for the functional $\Omega
\mapsto \lambda_k (\Omega)$ if and only if ${d\over d\varepsilon}
\lambda_{k,\varepsilon} {|_{\varepsilon = 0^-}}\leq 0 \leq {d\over
d\varepsilon} \lambda_{k,\varepsilon}|_{\varepsilon = 0^+}$ (or,
equivalently, $\lambda_{k,\varepsilon} \leq \lambda_{k,0} +
o(\varepsilon)$).

Similarly, if $\lambda_{k-1} < \lambda_k$, then, for
sufficiently small $\varepsilon$, $\lambda_{k,\varepsilon}\displaystyle =\min_{i \leq p}
 \Lambda_{i,\varepsilon}$ and ${d\over d\varepsilon} \lambda_{k,\varepsilon}
|_{\varepsilon = 0^+} \leq {d\over d\varepsilon}\ \lambda_{k,\varepsilon} {|_{\varepsilon = 0^-}}$.
\end{remark}
\bigskip

\begin{lemma}\label{digonalquadra}
Let $\lambda \in Sp_{_D}(\Delta_g , \Omega)$ be an eigenvalue of multiplicity $p$
of the Dirichlet Laplacian in $\Omega$ and let us
denote by $E_\lambda$ the corresponding
eigenspace. Let $\Omega_\varepsilon = f_\varepsilon (\Omega)$ be
an analytic deformation of $\Omega$ and let $(\Lambda_{i,
\varepsilon})_{i \leq p}$ 
 and $(\phi_{i, \varepsilon})_{i \leq p} \subset {\mathcal C}^\infty
(\Omega_\varepsilon)$ be as in
Lemma \ref{perturbation}. Then
$\Lambda_1':= {d\over
d\varepsilon}\Lambda_{1,\varepsilon}\big|_{\varepsilon=0} , \cdots
,\Lambda_p':= {d\over d\varepsilon}
\Lambda_{p,\varepsilon}\big|_{\varepsilon=0}$ are the eigenvalues
of the quadratic form $q_v$ defined on the space $E_\lambda
\subset L^2(\Omega,g)$ by
$$q_v(\phi) =- \int_{\partial\Omega} v\left({\partial \phi \over \partial \nu}\right)^2 v_{\bar g},$$
where $v=g\left({d\over d\varepsilon}f_\varepsilon\big|_{\varepsilon=0},\nu\right)$.
Moreover, the $L^2$-orthonormal basis $\phi_{1,0}, \cdots ,\phi_{p,0}$ diagonalizes $q_v$ on $E_\lambda$.
\end{lemma}

\begin{proof} 
For simplicity, we set $g_\varepsilon:= f^*_\varepsilon g$,
$\Delta':= {d\over d\varepsilon}
\Delta_{g_\varepsilon}\big|_{\varepsilon = 0}$, $\Lambda_i:=\Lambda_{i,0}$, $\phi_i:= \phi_{i,0}$ and $\Lambda'_i:=
{d\over d\varepsilon} \Lambda_{i,\varepsilon}\big|_{\varepsilon=0}$. From
$\Delta_{g_\varepsilon} (\phi_{i,\varepsilon})= \Lambda_{i,\varepsilon} (\phi_{i,\varepsilon})$, we deduce
$$\Delta' \phi_i + \Delta \phi'_i = \Lambda'_i \phi_i + \Lambda_i \phi'_i $$
We multiply by $\phi_j$ and integrate to get
$$ \int_\Omega \phi_j \Delta' \phi_i v_g + \int_\Omega \phi_j \Delta \phi'_i v_g
= \Lambda'_i \int_\Omega \phi_i \phi_j v_g + \lambda \int_\Omega \phi_i \phi'_j v_g.$$
Integration by parts gives ( since $\phi_{j}=\phi^{'}_{i}=0$ on $\partial \Omega$)) 
$$ \int_\Omega \phi_j \Delta \phi'_i v_g = \lambda \int_\Omega \phi_i \phi'_j v_{g}.$$
Hence,
$$ \int_\Omega \phi_j \Delta' \phi_i v_g = \Lambda'_i \int_\Omega \phi_i \phi_j v_g.$$
Therefore
$$ \int_\Omega \phi_j \Delta' \phi_i v_g = \Lambda'_i \int_\Omega \phi_i \phi_j v_g.$$
It follows that the $L^2$-orthonormal basis $\phi_{1}, \cdots ,\phi_{p}$
diagonalizes the quadratic form $\phi \to \int_\Omega \phi \Delta'
\phi v_g$ on $E_\lambda$, the corresponding eigenvalues being
$\Lambda_1', \cdots ,\Lambda_p'$. As we have seen in the
proof of the corollary \ref{vardir}, this last
quadratic form coincides with $q_v$ on $E_\lambda$.
\end{proof}

Any {\it volume-preserving} deformation $\Omega_\varepsilon=f_\varepsilon(\Omega)$ induces a function $v:=g({d \over d\varepsilon}f_\varepsilon\big|_{\varepsilon=0}, \nu)$ on $\partial \Omega$ satisfying $\int _{\partial \Omega} v \, v_{\bar g} =0$ (indeed, this last integral is equal up to a constant to ${d \over d\varepsilon}vol(\Omega_\varepsilon)\big|_{\varepsilon=0}$). 
In all the sequel, we will denote by $\mathcal A_0(\partial\Omega)$ the set of regular functions on $\partial \Omega$ such that $\int _{\partial \Omega} v \, v_{\bar g} =0$. The following elementary lemma will be useful in the proof of our main results. 

\begin{lemma} \label{defconstruct} Let $v\in \mathcal A_0(\partial\Omega)$. Then there exists an 
analytic volume-preserving deformation $\Omega_{\varepsilon}=f_{\varepsilon}(\Omega)$ so that $v= g(\frac{d}{d\varepsilon}f_{\varepsilon}\big |_{\varepsilon=0},\nu)$.
\end{lemma}

\begin{proof}
Let $U\subset M$ be an open neighborhood of $\bar{\Omega}$ and let $\tilde{v}$ and $\tilde{\nu}$ be smooth extensions to $U$ of $v$ and $\nu$ respectively. For
$\varepsilon$ sufficiently small, the map $ \varphi_{\varepsilon}(x)= {\rm exp}_{x}\,\varepsilon \, \tilde{v}(x)\, \tilde{\nu}(x)$ is a diffeomorphism from $\Omega$ to $\varphi_{\varepsilon}(\Omega)$. Moreover, since $(M,g)$ is real analytic, the curve $\varepsilon \rightarrow \varphi_{\epsilon}$ is analytic w.r.t. $\varepsilon$. The deformation $\varphi_{\varepsilon}(\Omega)$ is not necessarily volume-preserving. However, let $X$ be any analytic vectorfield on $U$ such
that $ \int_{\Omega} {\rm div} X v_{g} \ne 0$ and denote by $(\gamma_{t})_{t}$ the associated 1-parameter local group of diffeomorphisms. The function $ (t,\varepsilon) \mapsto F(t,\varepsilon)= vol(\gamma_{t}\circ \varphi_{\varepsilon}(\Omega))$ satisfies $ \frac{\partial}{\partial t}F(0,0)= \int_{\Omega} {\rm div} X v_{g} \ne 0 $. Applying the implicit function theorem in the analytic setting, we get the existence of a function $t(\varepsilon)$ depending analytically on $\varepsilon\in (-\eta,\eta)$, for some $\eta>0$ sufficiently small, such that $ F(t(\varepsilon),\varepsilon)=F(0,0)$,  $\forall \varepsilon\in (-\eta,\eta)$. The deformation $f_{\varepsilon}=\gamma_{t(\varepsilon)}\circ \varphi_{\varepsilon}$ is clearly analytic and volume-preserving. Moreover, one has
$$ t'(0)=\displaystyle{ -\frac {\frac{d}{d\varepsilon} vol(\varphi_{\varepsilon}(\Omega))\big |_{\varepsilon=0}}{{\frac{d}{dt}vol(\gamma_{t}(\Omega))\big |_{t=0}}}= 
-\frac{\int_{ \Omega}{\rm div} \tilde{v}\tilde{\nu} \, v_{{g}}}{\int_{\Omega}{\rm div}  X v_{{g}}}
=-\frac{\int_{\partial \Omega}v\, v_{\bar{g}}}{\int_{\partial \Omega}\left\langle X,\nu\right\rangle v_{\bar{g}}}}=0.$$
Therefore, $\forall x \in \partial \Omega$,
$$ \frac{d}{d\varepsilon}f_{\varepsilon}(x)\big|_{\varepsilon=0}= t'(0) X(x)+ \frac{d\varphi_{\epsilon}(x)}{d\varepsilon}\big|_{\varepsilon=0}= v(x) \nu(x).$$
\end{proof}


\subsection{Critical domains for the $k$-th eigenvalue of the Dirichlet Laplacian}

In all the sequel, we will denote by $\lambda_k$ the $k$-th
eigenvalue of the Dirichlet problem in $\Omega$ and by $E_k$ the
corresponding eigenspace.

In the following results, a special role is played by the eigenvalues $\lambda_k$ satisfying $\lambda_k >\lambda_{k-1}$ or $\lambda_k <\lambda_{k+1}$. This means that the index $k$ is the lowest or the highest one among all the indices corresponding to the same eigenvalue. Let us start with the following necessary condition to be satisfied by a locally minimizing or locally maximizing domain. Here, a local minimizer (resp. maximizer) for the $k$-th eigenvalue of the Dirichlet Laplacian is a domain $\Omega$ such that, for any volume-preserving deformation $\Omega_\varepsilon$, the function $\varepsilon \mapsto \lambda_ {k,\varepsilon}$ admits a local minimum (resp. maximum) at $\varepsilon =0$.

\begin{theorem}\label{lambda_k min}
Let $k$ be a natural integer such that $\lambda_k >\lambda_{k-1}$
(resp. $\lambda_k <\lambda_{k+1}$) and assume that $\Omega$ is a
local minimizer (resp. local maximizer) for the $k$-th eigenvalue
of the Dirichlet Laplacian. Then $\lambda_k$ is simple and the
absolute value of the normal derivative of its corresponding
eigenfunction is constant on $\partial \Omega$. That is, there
exists a unique (up to sign) function $\phi$ satisfying
$$\left\{
\begin{array}{l}
\Delta {\phi} = \lambda_k {\phi}\ \hbox{in}\ \Omega \\
\\
{\phi}=0 \ \ \hbox{on} \ \partial\Omega\\
\\
|{\partial \phi\over \partial\nu}| =1 \ \ \hbox{on}\ \partial\Omega.
\end{array}
\right.$$
\end{theorem}
\begin{proof}
Suppose that $\lambda_k > \lambda_{k-1}$ and let
$\Omega_\varepsilon = f_\varepsilon (\Omega)$ be a volume
preserving analytic deformation of $\Omega$. Let
$(\Lambda_{i,\varepsilon}) {_{i \leq p}}$ and
$(\phi_{i,\varepsilon})_{i \leq p}$ be families of eigenvalues and
eigenfunctions associated to $\lambda_k$ according to Lemma
\ref{perturbation}. Since $\Lambda_{i,0}= \lambda_k >
\lambda_{k-1}$, we have, for sufficiently small $\varepsilon$, for
continuity reasons,
$$\Lambda_{i,\varepsilon} > \lambda_{k-1, \varepsilon}.$$
Hence,
$$\Lambda_{i,\varepsilon} \ge \lambda_{k, \varepsilon}.$$
As the function $\varepsilon \mapsto \lambda_ {k,\varepsilon}$
admits a local minimum at $\varepsilon =0$ with $\Lambda_{i,0}=
\lambda_{k,0}=\lambda_k$, it follows that the differentiable
function $\varepsilon \mapsto \Lambda_ {i,\varepsilon}$ achieves a
local minimum at $\varepsilon =0$ and that ${d \over d\varepsilon}
\Lambda_ {i,\varepsilon}\big|_{\varepsilon =0}=0$. Applying Lemma
\ref{digonalquadra}, we deduce that the quadratic form $q_v$ is
identically zero on the eigenspace $E_k$, where $v= g({d \over
d\varepsilon} f_\varepsilon\big|_{\varepsilon =0}, \nu)$.
 The volume-preserving deformation being arbitrary, it follows that the form $q_v$
 vanishes on $E_k$ for any $v\in \mathcal A_0(\partial\Omega)$ (Lemma \ref{defconstruct}).
 Therefore, $\forall \phi \in E_k$ and $\forall v\in \mathcal A_0(\partial\Omega)$,
 we have $\int_{\partial\Omega} v \left({\partial \phi \over \partial \nu}\right)^2 v_{\bar g}=0$,
 which implies that ${\partial \phi \over \partial \nu}$ is locally constant on $\partial \Omega$
 for any $\phi \in E_k$. Now, if $\phi_1$ and $\phi_2$ are two eigenfunctions in $E_k$,
 one can find a linear combination $\phi = \alpha \phi_1+\beta \phi_2$
 so that $\partial\phi\over \partial \nu$ vanishes on at least one
connected component of $\partial \Omega$. We apply Holmgren uniqueness theorem (see for instance \cite[Theorem 2, p. 42]{Ra}, and recall that $(M,g)$ is assumed to be real analytic) to deduce that $\phi$ is identically zero in $\Omega$ and that $ \lambda_ k$ is simple.

To finish the proof, we must show that, $\forall \phi\in E_k$, $|{\partial \phi \over \partial \nu}|$
takes the same constant value on all the components of $\partial\Omega$. Indeed, let $\Sigma_1$
and $\Sigma_2$ be two distinct connected components of $\partial\Omega$
 and let $v\in \mathcal A_0(\partial\Omega)$ be the function given
 by $v=vol(\Sigma_2)$ on $\Sigma_1$, $v=-vol(\Sigma_1)$ on $\Sigma_2$
 and $v=0$ on the other components. Then the condition
 $\int_{\partial\Omega} v \left({\partial \phi \over \partial \nu}\right)^2 v_{\bar g}=0$
 implies that $\left({\partial \phi \over \partial \nu}\right)^2\big|_{\Sigma_1}=
 \left({\partial \phi \over \partial \nu}\right)^2\big|_{\Sigma_2}$.

 Of course, the same arguments work in the case $\lambda_k
 <\lambda_{k+1}$.
\end{proof}
The criticality of the domain $\Omega$ for the $k$-th eigenvalue
of Dirichlet Laplacian is
closely related to the definiteness of the quadratic forms $q_v$
introduced in Lemma \ref{digonalquadra} above, on the eigenspace
$E_k$. Indeed, we have the following

\begin{theorem}\label{quadradir}
Let $k$ be any natural integer.
\begin{enumerate}
\item If $\Omega$ is a critical domain for the $k$-th eigenvalue
of the Dirichlet Laplacian, then, $\forall v\in \mathcal A_0(\partial\Omega)$,
the quadratic form $q_v(\phi)
=- \int_{\partial\Omega} v \left({\partial \phi \over \partial \nu}\right)^2 v_{\bar g}$
is not definite on $E_{k}$.

\item Assume that $\lambda_k >\lambda_{k-1}$ or $\lambda_k <\lambda_{k+1}$
and that $\forall v\in \mathcal A_0(\partial\Omega)$, the quadratic form
$q_v(\phi) =- \int_{\partial\Omega} v \left({\partial \phi \over \partial \nu}\right)^2 v_{\bar g}$
is not definite on $E_{k}$, then $\Omega$ is a critical domain for
the $k$-th eigenvalue of the Dirichlet Laplacian.
\end{enumerate}

\end{theorem}
\begin{proof}
(1) Consider a function $v\in \mathcal A_0(\partial\Omega)$ and let
$\Omega_\varepsilon=f_\varepsilon(\Omega)$ be an analytic volume-preserving
deformation of $\Omega$ so that $v:=g({d \over
d\varepsilon}f_\varepsilon\big|_{\varepsilon=0}, \nu)$ (Lemma \ref{defconstruct}). Let
$(\Lambda_{i,\varepsilon}){_{i \leq p}}$ and
$(\phi_{i,\varepsilon})_{i \leq p}$ be families of eigenvalues and
eigenfunctions associated to $\lambda_k$ according to Lemma
\ref{perturbation}. As we have seen above, there exists two
integers $i\le p$ and $j\le p$ so that ${d\over d\varepsilon}
\lambda_{k,\varepsilon}{\big|_{\varepsilon=0^-}} = {d\over
d\varepsilon} \Lambda_{i,\varepsilon}{\big|_{\varepsilon=0}}$ and
${d\over d\varepsilon} \lambda_{k,\varepsilon}
{\big|_{\varepsilon=0^+}} = {d\over d\varepsilon}
\Lambda_{j,\varepsilon} {\big|_{\varepsilon=0}}$. The criticality
of $\Omega$ then implies that ${d\over d\varepsilon}
\Lambda_{i,\varepsilon}{\big|_{\varepsilon=0}} \times {d\over
d\varepsilon} \Lambda_{j,\varepsilon}{\big|_{\varepsilon=0}} \le
0$. Applying Lemma \ref{digonalquadra}, we deduce that the
quadratic form $q_v$ admits both nonnegative and nonpositive
eigenvalues on $E_k$ which proves Assertion 1.

(2) Assume that $\lambda_k >\lambda_{k-1}$ and let
$\Omega_\varepsilon=f_\varepsilon(\Omega)$ be a volume-preserving
deformation of $\Omega$. Let $(\Lambda_{i,\varepsilon}){_{i \leq
p}}$ and $(\phi_{i,\varepsilon})_{i \leq p}$ be families of
eigenvalues and eigenfunctions associated to $\lambda_k$ according
to Lemma \ref{perturbation}. As we have seen in Remark \ref{rem1},
we have, for sufficiently small $\varepsilon$,
$\lambda_{k,\varepsilon}\displaystyle =\min_{i \leq p}
\Lambda_{i,\varepsilon}$. Hence,
$${d\over d\varepsilon} \lambda_{k,\varepsilon}\big|_{\varepsilon = 0^+}
=\min_{i \leq p}{d\over d\varepsilon}\Lambda_{i,\varepsilon}\big|_{\varepsilon=0}$$
and
$${d\over d\varepsilon}\ \lambda_{k,\varepsilon} {\big|_{\varepsilon = 0^-}}
=\max_{i \leq p}{d\over
d\varepsilon}\Lambda_{i,\varepsilon}\big|_{\varepsilon=0}.$$ Now,
the non definiteness of $q_v$ on $E_k$ means that its smallest
eigenvalue is nonpositive and its largest one is nonnegative.
According to Lemma \ref{digonalquadra}, this implies that ${d\over
d\varepsilon} \lambda_{k,\varepsilon}\big|_{\varepsilon =
0^+}=\min_{i \leq p}{d\over
d\varepsilon}\Lambda_{i,\varepsilon}\big|_{\varepsilon=0}\le 0$
and ${d\over d\varepsilon}\ \lambda_{k,\varepsilon}
{\big|_{\varepsilon = 0^-}}=\max_{i \leq p}{d\over
d\varepsilon}\Lambda_{i,\varepsilon}\big|_{\varepsilon=0}\ge 0$
which implies the criticality of the domain $\Omega$.

The case $\lambda_k <\lambda_{k+1}$ can be handled similarly.
\end{proof}

The indefiniteness of $q_v$ for any $v\in \mathcal
A_0(\partial\Omega)$ can be interpreted intrinsically in the
following manner:

\begin{lemma}\label{quadranondef}
Let $k$ be a natural integer. The two following conditions are
equivalent:
\begin{itemize}
\item[(i)] $\forall v\in \mathcal A_0(\partial\Omega)$, the
quadratic form $q_v$ is not definite on $E_{k}$ \item[(ii)] there
exists a finite family of eigenfunctions $(\phi_i)_{i \leq
m}\subset E_k$ satisfying
$$\displaystyle\sum_{i=1}^{m} \left({\partial \phi_i\over \partial \nu}\right)
^2=1\;\; \hbox{on}\; {\partial\Omega}.$$
\end{itemize}
 \end{lemma}
\begin{proof} To see that $(ii)$ implies $(i)$, it suffices to notice that, for any $ v\in \mathcal A_0(\partial\Omega)$
$$\sum_{i\le m} q_v(\phi_i) =-\sum_{i\le m} \int_{\partial\Omega} v\left({\partial \phi_i\over \partial \nu}\right)^2 v_{\bar g} =- \int_{\partial\Omega} v v_{\bar g} =0.$$
Therefore, $q_v$ is not definite on $E_{k}$.

The proof of ``$(i)$ implies $(ii)$" uses arguments similar to
those used in the case of closed manifolds by Nadirashvili
\cite{N} and the authors \cite{EI2}. Let $K$ be the convex hull of
$\left\{\left({\partial\phi\over
\partial \nu}\right)^2,\phi \in E_k\right\}$ in ${\mathcal C}^\infty (\partial
\Omega)$. Then, we need to show that the constant function 1 belongs to $K$.

Let us suppose for a contradiction that $1 \not\in K$, then, from
the Hahn-Banach theorem (applied in the finite dimensional vector
space spanned by $K$ and $1$ and endowed with the $L^2(\partial
\Omega, \bar g)$ inner product), there exists a function $v \in
{\mathcal C}^\infty (\partial \Omega)$ such that
$\int_{\partial\Omega} v\ v_{\bar g} >0$ and, $\forall \phi \in
E_k$,
$$\int_{\partial\Omega} v\bigl({\partial\phi \over \partial \nu}\bigr)^2 v_{\bar g}\le 0.$$
Hence, the zero mean value function $v_o=v-{1\over vol(\partial\Omega)} \int_{\partial\Omega} v\ v_{\bar g}$
satisfies, $\forall \phi \in E_k$,
\begin{eqnarray}
\nonumber {} q_{v_0}(\phi) &=& - \int_{\partial\Omega} v_o\left({\partial\phi \over \partial \nu}\right)^2 v_{\bar g}\\
\nonumber {} &=& - \int_{\partial\Omega} v \left({\partial\phi \over \partial \nu}\right)^2 v_{\bar g}
+ {1\over vol(\partial\Omega)} \int_{\partial\Omega} v\ v_{\bar g} \int_{\partial\Omega}\left({\partial\phi \over \partial \nu}\right)^2
 v_{\bar g}\\
\nonumber {} &\ge&{1\over vol(\partial\Omega)} \int_{\partial\Omega} v\ v_{\bar g} \int_{\partial\Omega}\left({\partial\phi \over \partial \nu}\right)^2 v_{\bar g},
\end{eqnarray}
with $\int_{\partial\Omega}\left({\partial\phi \over \partial
\nu}\right)^2 v_{\bar g}>0$ for any non trivial Dirichlet
eigenfunction $\phi$ (due to Holmgren uniqueness theorem). In
conclusion, the function $ v_o\in \mathcal A_0(\partial\Omega)$ is
such that the quadratic form $q_{v_0}$ is positive definite on
$E_k$, which contradicts Condition $(i)$.
\end{proof}

A consequence of this lemma and Theorem \ref{quadradir} is the
following:
\begin{theorem}\label{lambda_k}
Let $k$ be any natural integer.
\begin{enumerate}
\item If $\Omega$ is a critical domain for the $k$-th eigenvalue of Dirichlet Laplacian, then there exists a
finite family of eigenfunctions $(\phi_i)_{i \leq m}\subset E_k$ satisfying
$\sum_{i=1}^{m} \left({\partial \phi_i\over \partial \nu}\right)
^2=1$ on ${\partial\Omega}$,
that is, $(\phi_i)_{i \leq m}$ are solutions of the following system
$$\left\{
\begin{array}{l}
\Delta {\phi_i} = \lambda_k {\phi_i} \ \hbox {in} \ \Omega, \, \ \forall i \leq m, \\
\\
{\phi_i}=0 \ \hbox {on} \ \partial \Omega,\ \forall i \leq m,\\
\\
\sum_{i=1}^{m} \left({\partial \phi_i\over \partial \nu}\right)^2 =1 \ \hbox {on} \ \partial \Omega.
\end{array}
\right.$$
\item Assume that $\lambda_k > \lambda_{k-1}$ or $\lambda_k < \lambda_{k+1}$ and that there exists a finite family of eigenfunctions $(\phi_i)_{i \leq m}\subset E_k$ such that $\sum_{i=1}^{m} \left({\partial \phi_i\over \partial \nu}\right)
^2$ is constant on ${\partial\Omega}$, then the domain $\Omega$ is critical for the k-th eigenvalue
of the Dirichlet Laplacian.
\end{enumerate}
\end{theorem}

\begin{corollary}\label{criticsimpledir}
Assume that $\lambda_k$ is simple. The domain $\Omega$ is critical
for the $k$-th eigenvalue of the Dirichlet Laplacian if and only
if the following overdetermined Pompeiu type system admits a
solution
 $$\left\{
\begin{array}{l}
\Delta {\phi} = \lambda_k {\phi}\ \hbox{in}\ \Omega \\
\\
{\phi}=0 \ \ \hbox{on} \ \partial\Omega\\
\\
|{\partial \phi\over \partial\nu}| =1 \ \ \hbox{on}\ \partial\Omega.
\end{array}
\right.$$

\end{corollary}


\subsection{Nonexistence of critical domains under metric variations}

In this paragraph, we point out the non consistency of the notion
of critical domains w.r.t. metric variations under Dirichlet boundary condition. Indeed, if $g_\varepsilon$ is an
analytic variation of the metric $g$, then we can associate to
each eigenvalue $\lambda_k$ of the Dirichlet problem in $\Omega$, analytic families
$(\Lambda_{i,\varepsilon})_{i \leq p}\subset {\mathbb R}$ and
$(\phi_{i,\varepsilon})_{i \leq p}\subset{\mathcal C}^\infty (\Omega)$
(where $p$ is the multiplicity of $\lambda_k$) satisfying,
for sufficiently small $\varepsilon$,
\begin{enumerate}
\item $\ (\phi_{i,\varepsilon})_{i \leq p}$ is $L^2(\Omega,g_\varepsilon)$
orthonormal.

\item $ \forall\ i \in \{1,\ldots,p\}$, $\Lambda_{i,o} = \lambda_k$.

\item $\forall\ i \le p$,
$\left\{
\begin{array}{l}
\Delta_{g_\varepsilon} \phi_{i,\varepsilon} = \Lambda_{i,\varepsilon} \phi_{i,\varepsilon} \;
 \hbox {in} \; \Omega\\
\\
\phi_{i,\varepsilon}=0 \; \hbox {on} \; \partial\Omega

\end{array}
\right.$ 
\end{enumerate}
Therefore $\lambda_{k,\varepsilon}$ 
admits a left sided and a right sided derivatives at
$\varepsilon=0$, and we can mimic Definition \ref{def1} to
introduce the notion of critical domain for the $k$-th eigenvalue
of Dirichlet problem w.r.t. volume-preserving
variations of the metric. Thanks to Proposition \ref{varmetric}
and using arguments similar to those used above (see also \cite{
EI2, N}), we can show that, if the domain $(\Omega, g)$ is
critical for the $k$-th eigenvalue of Dirichlet 
problem, then there exists a family of eigenfunctions
$\phi_1,\ldots,\phi_m \in E_k$ satisfying

\begin{equation}\label{8}
\sum_{i=1}^{m} d \phi_i \otimes d \phi_i = g.
\end{equation}

Now, if we consider only volume-preserving conformal variations
$g_\varepsilon$ of $g$ (that is $g_\varepsilon
=\alpha_\varepsilon\ g$ with $\int_{\Omega}
\alpha_\varepsilon^{n\over 2} v_g = Vol (\Omega,g))$, then the
necessary condition (\ref{8}) for $(\Omega,g)$ to be critical
w.r.t such variations becomes $\displaystyle\sum_{i=1}^{m}
\phi^2_i=1$ in $\Omega$. As the eigenfunctions of the Dirichlet
Laplacian vanish on the boundary $\partial \Omega$, this last
condition can never be fulfilled by functions of $E_k$. Thus, we
have the following:

\begin{proposition}\label{nonexistdir} There is no critical domain $(\Omega,g)$ for
the $k$-th eigenvalue of the Dirichlet Laplacian under conformal
volume-preserving variations of the metric $g$.
\end{proposition}


\section{Applications to the trace of the heat kernel }

This section deals with critical domains of the trace of the
heat kernel under Dirichlet boundary condition.

Recall that the heat kernel $H$ of $(\Omega,g)$ under the
Dirichlet boundary condition is defined 
to be the solution of the following parabolic problem:

$$\left\{
\begin{array}{l}
({\partial\over \partial t } - \Delta_y) H(t,x,y) =0\\
\\
H(0,x,y)=\delta_x\\
\\
\forall y\in \partial \Omega, \; H(t,x,y) =0 
\end{array}
\right.$$
Its trace is the function
$$Y(t) = \int_\Omega H(t,x,x) v_g$$
The relationship between this kernel and the spectrum of the Dirichlet
Laplacian is given by
$$ H(t,x,y)=\sum_{k\ge 1} e^{- \lambda_k t} \phi_k (x)\phi_k (y)$$
where $(\phi)_{k\ge1}$ is an $L^2(\Omega,
g)$-orthonormal family of eigenfunctions satisfying

$$\left\{
\begin{array}{l}
\Delta \phi_k=\lambda_k \phi_k\;\; \hbox{in}\; \Omega\\
\\
\phi_k =0 \;\; \hbox{on}\; \partial \Omega\\
\end{array}
\right.$$\\
and then,

\medskip

\begin{equation}\label{9}
 Y(t) = \sum_{k\ge 1} e^{- \lambda_k t}.
\end{equation}

\medskip

Let $\Omega_\varepsilon$ be a smooth deformation of $\Omega$ and
let $Y_\varepsilon (t) = \sum_{k\ge 1} e^{-
\lambda_{k,\varepsilon} t} $ be the corresponding heat trace
function. Unlike the eigenvalues, the function $Y_\varepsilon
(t)$ is always differentiable in
$\varepsilon$ and {\it the domain $\Omega$ will be said critical
for the trace of the heat kernel under the Dirichlet boundary condition at time $t$ if, for any volume-preserving deformation $\Omega_\varepsilon$ of $\Omega$, we have}
$${d\over {d \varepsilon }} Y_\varepsilon (t)\big| _{\varepsilon=0} =0$$
From the results of section 3 above, one can deduce the variation
formula for the heat trace. For this, we need to introduce the
mixed second derivative $d_SH(t)|_x$ of $H$ at the point $x$
defined as the smooth 2-tensor given by
$$d_SH(t)|_x(X,X) = {\partial^2 \over {\partial\alpha\partial\beta}}
H(t,c(\alpha),c(\beta))\big|_{\alpha=\beta=0},$$
where $c$ is a curve in $\Omega$ such that $c(0)=x$ and
$\dot{c}(0)=X$. It is easy to check that
$$d_SH(t)= \sum_{k\ge1}e^{- \lambda_k t}d\phi_k \otimes d\phi_k$$
\begin{theorem}\label{varnoy}
Let $\Omega_\varepsilon=f_\varepsilon (\Omega)$ be a volume-preserving deformation of $\Omega$. We have, $\forall t>0$,
$${d\over {d \varepsilon }} Y_\varepsilon (t)\big| _{\varepsilon=0}
=-t\int_{\partial\Omega} v\,  d_SH(t)(\nu,\nu) v_{\bar g} ={t\over
2} \int_{\partial\Omega} v \, \Delta H(t,x,x) v_{\bar g}$$
where $v=g({d\over {d\varepsilon}}f_\varepsilon \big|
_{\varepsilon =0},\nu)$.
\end {theorem}

\begin{proof}
The proof can be derived from the first variation formula of the
heat kernel that can be found in the paper of Ray
and Singer \cite[Proposition 6.1]{RS}. Nevertheless, at least in the case where the ambient manifold is real analytic, the formula of Theorem \ref{varnoy} can be obtained as an immediate consequence of Hadamard's type formula of Section 2,  thanks to the relation
(\ref{9}) above. Indeed, in this manner we obtain,
$\forall t>0$,

$${d\over {d \varepsilon }} Y_\varepsilon (t)\big|_{\varepsilon=0}
= -t \sum_{k\ge1}e^{- \lambda_k t}\int_{\partial\Omega} v
\left({\partial\phi_k\over{\partial\nu}}\right)^2v_{\bar g} $$
where $(\lambda_k,\phi_k)$ are as above. To
get the desired formula for $Y_\varepsilon (t)$ it suffices to
notice that
$$ d_SH(t)(\nu,\nu)=\sum_{k\ge1}e^{- \lambda_k t}d\phi_k \otimes
d\phi_k (\nu,\nu) =\sum_{k\ge1}e^{- \lambda_k t}
\left( {\partial\phi_k\over{\partial\nu}} \right) ^2.$$

\end{proof}

An immediate consequence is the following

\begin{corollary}\label{criticH}
The following conditions are equivalent:

\begin{itemize}
\item [(i)] The domain $\Omega$ is critical for the trace of the
Dirichlet heat kernel at the time $t$ under volume-preserving domain deformations, 
\item [(ii)] $\Delta H(t,x,x) $ is constant on the boundary
$\partial\Omega$, 
\item [(iii)] for any positive integer $k$ and
any $L^2(\Omega, g)$-orthonormal basis $\phi_1, \cdots \phi_p$ of
the eigenspace $E_k$ of $\lambda_k$, $\sum_{i\le p}
\left({\partial\phi_i\over{\partial\nu}}\right)^2$ is constant on $\partial\Omega$.
\end{itemize}

\end{corollary}

Recall that if $\rho$ is an isometry of $(\Omega,g)$, then,
$\forall x \in \Omega$ and $\forall t>0$, $H(t,\rho(x), \rho
(x))=H(t,x,x)$. In
particular, if $\Omega$ is a ball of ${\mathbb R}^n$ endowed with
a rotationally symmetric Riemannian metric $g$ given in polar
coordinates by $g= a^2(r) dr^2 + b^2(r) d\sigma^2$, where
$d\sigma^2$ is the standard metric of the unit sphere ${\mathbb
S}^{n-1}$, then $H(t,x,x)$ should be radial (that
is depend only on the parameter $r$). Therefore, the function
$\Delta H(t,x,x)$ is also radial and then
it is constant on the boundary of the ball.

\begin{corollary}\label{ballsym}
Let $g$ be a a rotationally symmetric Riemannian metric on ${\mathbb R}^n$. The geodesic balls centered at the origin are critical domains for the trace of the
Dirichlet heat kernel under volume-preserving
domain deformations.
\end {corollary}

In particular, geodesic balls of Riemannian space forms are
critical for the trace of the Dirichlet heat
kernel under volume-preserving domain deformations.

The Minakshisundaram-Pleijel asymptotic expansion of the trace of
the heat kernel can also informs us about the geometric properties
of extremal or critical domains. Indeed, it is well known that
there exists a sequence $(a_{i})_{i \in {\mathbb N} }$ of real numbers such that for sufficiently small $t>0$,
we have:

$$ Y(t)= (4\pi t)^{- n \over 2} \sum_{k \geq 0}
a_{k}t^{k \over 2}$$ with (see for instance
\cite{BBG, BG}):
$$ a_{0}= vol(\Omega ,g),$$
$$a_{1}= -\frac {\sqrt{\pi}}{2} vol(\partial \Omega,\bar g),$$
$$a_{2}= {1 \over 6} \left\{ \int_\Omega scal_{g} v_{g}
+2 \int _{\partial \Omega} tr A\, v_{\bar g} \right \},$$
$$a_{3}= {\sqrt \pi \over 192 } \left\{ \int_{\partial \Omega}
\left ( -16 scal_{g} -7 (trA)^{2}+10 |A|^{2}+8 \rho_{g} (\nu,\nu)
\right ) v_{\bar g} \right\},$$
where $scal_g$ and $\rho_g $ are
respectively the scalar and the Ricci curvatures of $(\Omega,g)$,
$A$ is the shape operator of the boundary $\partial \Omega$ (i.e
$\forall X \in T\partial\Omega, \, A(X)=D_{X}\nu$) and $tr A$ is
the trace of $A$ (i.e $(n-1)$-times the mean curvature of
$\partial\Omega$).

An immediate consequence of these formulae is the following: 
Suppose that for any domain $\Omega'$ having the same volume as
$\Omega$, we have $Y_{\Omega'}(t)\leq
Y_{\Omega}(t)$,$\forall t>0$, then
$vol\,\partial\Omega' \geq vol \,\partial \Omega$. Consequently, we
have

\begin{proposition}\label{isoper}
If the domain $\Omega$ maximizes $Y$ at any time $t>0$ among
all the domains of the same volume, then $\Omega$ is a solution of
the isoperimetric problem in $(M,g)$, that is, $\forall \Omega'
\subset M$ such that $vol\Omega=vol\Omega'$, we have $vol\partial
\Omega' \geq vol \partial \Omega$.

\end{proposition}

Another consequence of the Minakshisundaram-Pleijel asymptotic
expansion is the following

\begin{theorem}\label{hconst}
If the domain $\Omega$ is a critical domain of the trace of the Dirichlet
heat kernel at any time $t>0$, then $\partial \Omega$ has constant
mean curvature. If in addition the Ricci curvature (resp. the
sectional curvature) of the ambient space $(M,g)$ is constant in a
neighborhood of $\Omega$, then $tr(A^{2})$ (resp. $tr(A^{3})$) is
constant on $\partial\Omega$.
\end{theorem}

\begin{proof}
Let $\Omega_{\varepsilon}=f_{\varepsilon}(\Omega)$ be a volume-preserving variation of $\Omega$ and let us denote for any
$\varepsilon$ by $(a_{i,\varepsilon})_{i \geq 0}$ the coefficients
of the asymptotic expansions of $Y_{\varepsilon}(t)$. Since $ {d
\over d\varepsilon}Y_{\varepsilon}(t) \big|_{\varepsilon=0}=0$, we
have for any $i \geq 0$, $ {d \over d\varepsilon }
a_{i,\varepsilon} \big|_{\varepsilon =0}=0$ (see for instance
\cite{GS} for an analytic justification for this last assertion).
In particular, ${d \over d\varepsilon}
vol(\partial\Omega_{\varepsilon})\big|_{\varepsilon=0}=0$ for any
volume-preserving variation of $\Omega$. This property is known to
be equivalent to the fact that the mean curvature of $\partial
\Omega$ is constant (see for instance \cite{R}).

Now, let us suppose that the Ricci curvature of $(M,g)$ is
constant in a neighborhood of $\Omega$, then for any small
$\varepsilon$, we have:
$$ \begin{aligned}
a_{2,\varepsilon} &= {1 \over 6} \left\{ scal_{g} vol(\Omega_{\varepsilon})
+2 \int _{\partial \Omega_{\varepsilon}} (tr A_{\varepsilon}) \, v_{\bar g} \right \}\\
 &={1 \over 6} \left\{ scal_{g} vol(\Omega) +2 \int _{\partial \Omega_{\varepsilon}}
 (tr A_{\varepsilon}) \, v_{\bar g}. \right \}
\end{aligned}$$
Hence, we have (see for instance \cite{R}):
$$ \begin {aligned}
{d \over d\varepsilon} \int _{\partial \Omega_{\varepsilon}} (tr A_{\varepsilon}) \,
v_{\bar g} \big|_{\varepsilon =0}& = \int_{\partial\Omega}
\left( \Delta_{\bar g}v-\rho(\nu,\nu)v-(trA^{2})v \right) v_{\bar g}\\
 &+ \frac12 \int_{\partial\Omega}trA \left(div_{\bar g}V^{T}+v\, trA \right) v_{\bar g},\\
\end{aligned}$$
where $V= {df_{\varepsilon} \over d\varepsilon} \big|_{\varepsilon =0}= v \,
\nu + V^{T}$ on the boundary $\partial \Omega$.\\
Since $\int_{\partial\Omega}v\, v_{\bar g}=0$ and $trA$ and
$\rho(\nu,\nu)$ are constant on $\partial \Omega$, we have:
$$ {d \over d\varepsilon} a_{2,\varepsilon} \big|_{\varepsilon=0}=
{1 \over 3}\int_{\partial\Omega}(trA^{2})v\, v_{\bar g}=0.$$ It
follows that $trA^{2}$ is constant on $\partial \Omega$.

 As before, we have
$${d \over d\varepsilon} a_{3,\varepsilon} \big|_{\varepsilon=0}
={ {\sqrt \pi} \over 192} \left( -7\,
{d \over d\varepsilon}\big|_{\varepsilon=0}\int_{\partial\Omega_{\varepsilon}}
(trA_{\varepsilon})^{2} \, v_{\bar g}+
10\,{d \over
d\varepsilon}\big|_{\varepsilon=0}\int_{\partial\Omega_{\varepsilon}}trA_{\varepsilon}^{2}
\, v_{\bar g}\right)$$ but,
$$ \begin{aligned}
{d \over d\varepsilon}\int_{\partial\Omega_{\varepsilon}}(trA_{\varepsilon})^{2} \,
v_{\bar g}\big|_{\varepsilon=0} & =2 \int_{\partial\Omega}
trA\left( \Delta_{\bar g}v-\rho(\nu,\nu)v-(trA^{2})v \right) v_{\bar g}\\
 &+ \frac12\int_{\partial\Omega}(trA)^{2} \left(div_{\bar g}V^{T}+v \, trA \right) v_{\bar g}\\
 &=0
\end{aligned}$$
since $trA$, $trA^{2}$ and $\rho(\nu,\nu)$ are constants. Thus,
$$ {d \over d\varepsilon} a_{3,\varepsilon} \big|_{\varepsilon=0}
={ 10{\sqrt \pi} \over 192}\;{d \over d\varepsilon}\big|_{\varepsilon=0}
\int_{\partial\Omega_{\varepsilon}}trA_{\varepsilon}^{2} \, v_{\bar g}.$$
After some straightforward but long computations we obtain, using
the fact that the sectional curvature is constant in a
neighborhood of $\Omega$, and that $trA$ and $trA^{2}$ are
constant,
$${d \over d\varepsilon} a_{3,\varepsilon} \big|_{\varepsilon=0}
= c\,\int_{\partial\Omega}trA^{3}\,v \, v_{\bar g}=0,$$ where $c$
is a constant. This proves that $trA^{3}$ is constant.
\end{proof}

 Alexandrov's Theorem \cite{Al} shows that in the Euclidean space, the geodesic spheres
are the only embedded compact hypersurfaces of constant mean
curvature. This theorem was extended to hypersurfaces of the hyperbolic space and the standard hemisphere(see \cite{MR}). Since the boundary of a critical domain
of the trace of the heat kernel is an embedded hypersurface of
constant mean curvature, we have the

\begin{corollary}\label{balls}
Let $(M,g)$ be one of the following spaces:
\begin{itemize}
\item The Euclidean space. \item The Hyperbolic space. \item The
standard Hemisphere.
\end{itemize}
Then a domain $\Omega$ of $(M,g)$ is critical for the trace of the
Dirichlet heat kernel if and
only if $\Omega$ is a geodesic ball.
\end{corollary}

\bigskip

\noindent \textbf{ Acknowledgment}\\
The authors would like to Thank Professors Bernard Helffer and
Peter Gilkey for valuable discussions. They also thank the referee for pointing out some mistakes in the first version of the paper and for valuable comments

\end{document}